\documentclass[12pt,a4paper]{amsart}

\usepackage{fourier}
\usepackage{ulem}
\usepackage{amsthm}
\usepackage{mathrsfs}
\usepackage{amsmath}
\usepackage[dvipsnames]{color} 
\usepackage[all]{xy}
\usepackage{xypic}
\usepackage[hidelinks]{hyperref}
\hypersetup{
  pdftitle   = {Deformations of Noncompact Calabi-Yau threefolds},
  pdfauthor  = {Gasparim et al.},
  pdfcreator = {\LaTeX\ with package \flqq hyperref\frqq}
}

\DeclareMathAlphabet{\mathpzc}{OT1}{pzc}{m}{it}

\newtheorem{theorem}{Theorem}[section]
\newtheorem*{theorem*}{Theorem}
\newtheorem{proposition}[theorem]{Proposition}
\newtheorem{lemma}[theorem]{Lemma}
\newtheorem*{lemma*}{Lemma}
\newtheorem{corollary}[theorem]{Corollary}
\newtheorem*{corollary*}{Corollary}

\newtheorem*{conjecture*}{Conjecture}

\theoremstyle{definition}
\newtheorem{definition}[theorem]{Definition}
\newtheorem{example}[theorem]{Example}
\newtheorem{notation}[theorem]{Notation}
\theoremstyle{remark}
\newtheorem{remark}[theorem]{Remark}

\newcommand{\ce}{\mathrel{\mathop:}=}

\DeclareMathOperator{\Ext}{Ext}

\DeclareMathOperator{\Tot}{Tot}
\DeclareMathOperator{\Proj}{Proj}
\DeclareMathOperator{\HH}{\mathrm{H}}

\numberwithin{equation}{section}

\begin{document}

\title[Deformations of Noncompact Calabi--Yau threefolds]{Deformations of Noncompact Calabi--Yau threefolds}

\author{E. Gasparim, T. K\"oppe, F. Rubilar,  {\tiny and} B. Suzuki}

\begin{abstract}
  We describe deformations of noncompact Calabi--Yau threefolds
  $W_k \ce \Tot(\mathcal{O}_{\mathbb{P}^1}(-k) \oplus \mathcal{O}_{\mathbb{P}^1}(k-2))$
  for $k=1,2,3$. We compute deformations concretely by calculations of $\HH^1(W_k, TW_k)$ via \v{C}ech cohomology.
  We show that for each $k=1,2,3$ the associated structures
  are qualitatively different, and we also comment on their difference from the analogous
  structures of simpler noncompact twofolds $\Tot(\mathcal{O}_{\mathbb{P}^1}(-k))$.
 \end{abstract}

\maketitle

\section{Motivation}
\noindent
 Our motivation to study deformations of Calabi--Yau threefolds comes from mathematical physics.
 In fact, deformations of complex structures of Calabi--Yau threefolds enter as terms of  the integrals defining
 the action of the theories of Kodaira--Spencer gravity \cite{B}. As we shall see, in general our 
 threefolds will have infinite-dimensional deformation spaces, thus allowing for rich applications.
  
We consider
smooth  Calabi--Yau threefolds $W_k$ containing a line $\ell \cong \mathbb{P}^1$.
For the applications we have in mind for future work it will be useful to observe the effect of contracting 
the line to a singularity. 
The existence of a contraction of $\ell$ imposes heavy restrictions on
the normal bundle \cite{ji}, namely $N_{\ell/W}$ must be isomorphic to
one of
\[ \text{(a) \ } \mathcal{O}_{\mathbb{P}^1}(-1) \oplus \mathcal{O}_{\mathbb{P}^1}(-1)
   \text{ ,}\quad \text{(b) \ } \mathcal{O}_{\mathbb{P}^1}(-2) \oplus \mathcal{O}_{\mathbb{P}^1}(0)
   \text{ , \ or \ } \text{(c) \ } \mathcal{O}_{\mathbb{P}^1}(-3) \oplus \mathcal{O}_{\mathbb{P}^1}(+1)
   \text{ .} \]
Conversely, Jim\'enez \cite{ji} states that if $\mathbb{P}^1 \cong \ell \subset W$
is any subspace of a smooth threefold $W$ such that $N_{\ell/W}$ is
isomorphic to one of the above, then:
\begin{itemize}
\item in (a) $\ell$ always contracts, 
\item in (b) either $\ell$ contracts or it moves, and 
\item in case (c) there exists an example in which $\ell$ does not contract
      nor does any multiple of $\ell$ (i.e.\ any scheme supported on $\ell$) move.
\end{itemize}

$W_1$ is the space appearing in the {basic flop}. Let $X$ be the cone
over the ordinary double point defined by the equation $xy - zw = 0$
on $\mathbb{C}^4$. The basic flop is described by the diagram:
\begin{equation}\label{eq.flop}
 \xygraph{  !{<0cm,0cm>;<1cm,0cm>;<0cm,1cm>::}
   !{(0,1.2)}*+{W}="t"
   !{(1,0)}*+{W_1^+}="l"
   !{(-1,0)}*+{W_1^-}="r"
   !{(0,-1.2)}*+{X}="b"
   !{(0,0)}="m"
   "t":"l"^{p_2} "t":"r"_{p_1} 
   "r":"b"_{\pi_1} "l":"b"^{\pi_2} "l":@{<--}"r"|\hole "t":"b"
   } 
\end{equation}
Here $W \ce W_{x,y,z,w}$ is the blow-up of $X$ at the vertex
$x=y=z=w=0$, $W_1^- \ce Z_{x,z}$ is the small blow-up of $X$ along
$x=z=0$ and $W_1^+ \mathrel{\mathop:}= Z_{y,w}$ is the small blow-up of $X$ along
$y=w=0$. 
The basic flop is the rational map from $W^{-}$ to
$W^{+}$. It is famous in algebraic geometry for being the first case of 
a rational map that is not a blow-up. 
 
 Thus, 
we will focus on  the Calabi--Yau cases 
\[ W_k \ce \Tot\bigl(\mathcal{O}_{\mathbb{P}^1}(-k) \oplus \mathcal{O}_{\mathbb{P}^1}(k-2)\bigr) \text{ for $k=1,2,3$.} \]
We will also consider surfaces of the form
\[ Z_k \ce \Tot\bigl(\mathcal O_{\mathbb P^1}(-k)\bigr) \]
for comparison in Sections \ref{sec.surfdeform} and \ref{sec.surfresult}.

\section{Statements of results}
\noindent
 We describe deformations of 
 complex surfaces and threefolds which are the total  spaces of vector bundles on the complex projective line $\mathbb{P}^1$.  We focus on the case of Calabi--Yau threefolds. 
 Even though there is no well established theory of deformations for noncompact manifolds, 
 we obtain deformations 
working by analogy with Kodaira theory for the compact case, see \cite{Ko}. Namely, we 
 calculate cohomology with  coefficients in the tangent bundle, and then proceed 
 to show that the directions of infinitesimal deformations parametrized by first cohomology are integrable, see 
 Section \ref{kod}.
 
In the case of  surfaces $Z_k$, with $ k>0$,  we prove that the deformations of the surfaces $Z_k$, described in \cite {BG}, can be obtained from the deformations
 of the Hirzebruch surfaces  $\mathbb{F}_k$, Lemma \ref{ZtDeformations}.

  Our results on deformations of the threefolds $W_k$ are as follows. We show that $W_1$ is formally rigid,
  Theorem \ref{W_1isrigid}, whereas $W_2$ has an infinite-dimensional deformation space,
   Theorem \ref{W_2iddef}. 
  Furthermore, we exhibit  a deformation  $\mathcal{W}_2$ of $W_2$  which turns out to be a non-affine manifold, a
  very different  case  from that of surfaces $Z_k $, $k>0$, where all the deformations are affine varieties. Finally, we give 
  an infinite-dimensional  family of deformations of $W_3$ which is not universal, but is semiuniversal, Corollary \ref{W_3sunivnouniv}. 
  The case $W_3$ is quite different from $W_1$, $W_2$, or the surfaces. 
  The tools used here to describe deformation spaces were not sufficient for $W_3$, therefore we must look for more effective techniques. 
  The cases $k\geq 3$  present similar features; we will continue their study in future work.

\section{Deformations of noncompact manifolds}\label{kod}

In this section we describe our methods to find infinitesimal deformations of  noncompact 
manifolds. 
When looking for deformations of noncompact manifolds one needs to keep in 
mind the caveat that cohomology calculations are generally not enough to decide 
questions of existence of infinitesimal deformations, as the following example illustrates. 

\begin{example} \label{ballico} Edoardo Ballico gave us the following illustration that cohomological rigidity does not imply absence of deformations. 

We consider deformations of $X= \mathbb C$. Clearly $\HH^1(X, TX)=0$.
However, there do exist nontrivial deformations of $X$ as the following family shows. 

Consider  $ \pi \colon  \mathbb P^1\times D \rightarrow D$  with $D$ any smooth manifold
 (even $\mathbb P^1$ or a disc) and a specific $ o\in D$.
Take $s_\infty \colon  D \to \mathbb P^1 \times D$ the section of $\pi$ defined by 
$$s_\infty(x) =(\infty, x) $$
then take another section $s$ of $\pi$ with 
$$s(o) =(\infty,o), \qquad s(x) =(a_x,x)$$ with $a_x \in \mathbb C = \mathbb P^1 \setminus \{\infty\}$.
Take as the total space for our family $Y $ as $\mathbb P^1 \times D$ minus the images of the two sections.
Then we obtain a deformation of $\mathbb C$ in which at all points of $D \setminus  \{o\}$ you have $\mathbb C\setminus \{0\}$.
Thus not a trivial deformation in any reasonable sense.
\end{example}

Hence, vanishing of cohomology does not imply nonexistence of deformations. Nevertheless, cohomology calculations are useful to find deformations. In this work by deformation we mean the following:
\begin{definition}\label{deformation}
	A {\it deformation} of a complex manifold $X$ is a holomorphic fiber bundle $\tilde{X}\stackrel{\pi}{\rightarrow}D$, where $D$ is a complex disc centered at $0$ (possibly a vector space, possibly infinite dimensional), satisfying:
	\begin{itemize}
		\item $\pi^{-1}(0)=X$,
		\item $\tilde{X}$ is trivial in the $C^{\infty}$ (but not necessarily in the holomorphic) category.
	\end{itemize}
\begin{remark}
	Our choice for the dimension of $D$ is $n=h^1(X,TX)$ whenever possible. The case $n=0$ corresponds to the following definition.
\end{remark}
\end{definition}

\begin{definition}\label{formal}
We call a manifold $X$ {\it formally rigid} when $\HH^1(X, TX)=0$.
\end{definition}

We show in \ref{w1} that $W_1$ is formally rigid. 

\begin{definition}\label{rigid}
We call a manifold $X$ {\it rigid} if any deformation 
$\tilde{X}\stackrel{\pi}{\rightarrow}D$ 
is biholomorphic to the trivial bundle $X \times D \to D$.
\end{definition}

In general, formally rigid does not imply rigid. 
With Definition \ref{deformation} we do not claim to solve the problem that a manifold $X$ does not deform under the condition $\HH^1(X,TX)=0$, however we eliminate some unwanted cases such as the one in Example \ref{ballico}.

Observe that the deformations considered in \cite{BG} satisfy Definition \ref{deformation}, hence maintain the $C^\infty$ type of the manifold. Moreover, for these surfaces, all deformations are affine.

We show that $\HH^1(W_2,TW_2) \neq 0$ and then prove that directions of deformations parametrized 
by such cohomology are integrable by explicitly constructing families. 
The details for other threefolds will 
remain for future work. 

Note that since $X$ is covered by 2 affine (Stein) open sets, second cohomologies with coherent coefficients vanish, 
hence there are no obstructions to deformations. 

\section{Comparison with  the deformation theory of  surfaces}\label{sec.surfdeform}
\noindent
Deformations of the surfaces $Z_k$ were described in \cite{BG}.
It turned out rather interestingly that the results we obtained for threefolds 
are not at all analogous to the ones for surfaces. 

Regarding applications to mathematical physics, the deformations of 
surfaces  turned out rather disappointing, because  instantons on $Z_k$ disappear 
under a small deformation of the base \cite[Thm.\thinspace 7.3]{BG}.
This resulted from the fact that deformations of $Z_k$ are affine varieties. 
The case of threefolds is a lot more promising, since for $k>1$,
$W_k$ has deformations which are not affine.

Nevertheless, deformations of the surfaces $Z_k$ turned out to have  an interesting application
to a question motivated by the Homological Mirror Symmetry conjecture. 
\cite[Sec.\thinspace 2]{BBGGS} showed that the adjoint orbit of $\mathfrak{sl}(2, \mathbb C)$ 
has the complex structure of the nontrivial deformation of $Z_2$ and used this structure to construct a 
Landau--Ginzburg model that does not have projective mirrors. Further applications to mirror symmetry give
us another motivation to study deformation theory for Calabi--Yau threefolds.

\section{$Z_k$, their bundles and deformations}\label{sec.surfresult}
\noindent
In this section we obtain properties about the surfaces $Z_k$ that will be used in the 
development of the theory of threefolds.
 
\subsection{A holomorphic bundle on \texorpdfstring{$Z_{(-1)}$}{Z\_(-1)} that is not algebraic}\label{bundleonZ1}

By definition  $Z_{(-1)}=\Tot(\mathcal{O}_{\mathbb{P}^1}(+1))$, and in
canonical coordinates $Z_{(-1)}=U\cup V$, where $U=\{(z,u)\}$ and  $V=\{(\xi,v)\}$, 
$U\cap V\cong \mathbb C^{\ast}\times \mathbb C$,
with change of coordinates given by:
\[ \boxed{(\xi,v)\mapsto(z^{-1},z^{-1}u)} \]

\begin{notation}
We denote by $\mathcal{O}_{Z_{(-1)}}(-j)=p^\ast(\mathcal{O}_{\mathbb{P}^1}(-j))$ the pullback bundle of $\mathcal{O}_{\mathbb{P}^1}(-2)$, where $p\colon Z_{(-1)} \to \mathbb{P}^1$ is the natural projection. 
\end{notation}

\begin{lemma}\label{cohZ(-1)}
$\HH^1(Z_{(-1)},\mathcal{O}_{Z_{(-1)}}(-2))$ is infinite-dimensional over $\mathbb{C}$. It consists of holomorphic functions of the form
\[
\sum_{l \leq -1} \sum_{i \geq 0} a_{li}z^lu^i
\]
such that $l+i+2 > 0$.
\end{lemma}

\begin{proof}
A 1-cocycle  $\sigma$  can be written in the form 
\[ \sigma=\sum_{i=0}^{+\infty}\sum_{l=-\infty}^{+\infty}\sigma_{i,l}z^lu^i \text{ .} \]
Since monomials containing nonnegative powers of  $z$ are holomorphic in $U$, these are coboundaries, thus
\[ \sigma\sim \sum_{i=0}^{+\infty}\sum_{l=-\infty}^{-1}\sigma_{i,l}z^lu^i \text{ ,} \]
where $\sim$ denotes cohomological equivalence.
Changing coordinates, we obtain
\[ T\sigma=z^{2}\sum_{i=0}^{+\infty}\sum_{l=-\infty}^{-1}\sigma_{i,l}z^lu^i = 
\sum_{i=0}^{+\infty}\sum_{l=-\infty}^{-1}\sigma_{i,l}z^{l+2}u^i  = \sum_{i=0}^{+\infty}\sum_{l=-\infty}^{-1}\sigma_{i,l}\xi^{-l-2-i}v^i  \text{ ,} \]
where terms satisfying $-l-2-i \geq 0$ are holomorphic on $V$.

Thus, the nontrivial terms on  $\HH^1(Z_{(-1)},\mathcal{O}_{Z_{(-1)}}(-2))$ are 

\[
\begin{array}{ccccc}
z^{-1} & z^{-1}u & z^{-1}u^2 & z^{-1}u^3 & \cdots \\
 & z^{-2}u & z^{-2}u^2 & z^{-2}u^3 & \cdots \\
 &  & z^{-3}u^2 & z^{-3}u^3 & \cdots \\ 
 &  &  & z^{-4}u^3 & \cdots \\
 & & & & \ddots
\end{array}
\]

\end{proof}

\begin{proposition}\label{nonalgebraicZ(-1)}
The bundle $E$ over $Z_{(-1)}$ defined  in canonical coordinates by the matrix 
\begin{equation}
 \left[\begin{array}{cc}
z^{1} & z^{-1}e^u\\
0 & z^{-1}
\end{array}\right] 
\label{matrix}
\end{equation}
is holomorphic but not algebraic.
\end{proposition}

\begin{proof}
This bundle $E$ 
can be represented by the element 
$$ z^{-1} e^u \in\Ext^1(\mathcal{O}_{Z_{(-1)}}(1),\mathcal{O}_{Z_{(-1)}}(-1))
\simeq  \HH^1(Z_{(-1)},\mathcal{O}_{Z_{(-1)}}(-2)).$$
We have 
\begin{equation}
\left[\begin{array}{cc}
z^{1} & z^{-1}e^u\\
0 & z^{-1}
\end{array}\right]=\left[\begin{array}{cc}
z^{1} & z\sigma\\
0 & z^{-1}
\end{array}\right]
\label{eq10}
\end{equation}
 with 
 $z^{-2} e^u = \sigma \in \HH^1(Z_{(-1)},\mathcal{O}_{Z_{(-1)}}(-2)),$ see \cite[p.\thinspace 234]{Har}. 
 Observe that
\begin{eqnarray*}
z^{-2} e^u &=& z^{-2} \left(1 + u + \frac{u^2}{2} + \dotsb + \frac{u^n}{n!} + \dotsb \right) \\
           &=& z^{-2} + \underbrace{z^{-2} \left(u + \frac{u^2}{2} + \frac{u^{3}}{6} + \dotsb +
               \frac{u^n}{n!} + \dotsb\right)}_{(\gamma)} \text{ ,}
\end{eqnarray*}
where the monomials in $\gamma$
represent pairwise distinct nontrivial classes in 
 $\HH^1(Z_{(-1)},\mathcal{O}_{Z_{(-1)}}(-2))$ as shown in Lemma \ref{cohZ(-1)}. 
 Consequently, the class $z\sigma \in \Ext^1(\mathcal{O}_{Z_{(-1)}}(1),\mathcal{O}_{Z_{(-1)}}(-1))$  
 corresponding to the bundle $E$  cannot be represented by a polynomial, hence 
 $E$ is holomorphic but not algebraic.
\end{proof}

\begin{corollary}\label{W3nonalg}
The threefold $W_3$ has holomorphic bundles that are not algebraic.
\end{corollary}

\begin{proof}  Consider the map  $p\colon W_3 \rightarrow Z_{(-1)}$ given by projection on the  
first and third coordinates, that is, in canonical coordinates  as in (\ref{W3}) we see $Z_{(-1)}$ as 
cut out inside $W_3$ by the equation $u_1=0$. Then the pullback bundle $p^*E$ is holomorphic but 
not algebraic on $W_3$. In fact, the same proof works as in Proposition \ref{nonalgebraicZ(-1)}.
\end{proof}

\subsubsection{A similar bundle on $Z_1$}
It is instructive to verify the result of defining a bundle 
by the same matrix, but  over the surface $Z_1$ instead.
Recall that  $Z_1=U\cup V$,  with change of coordinates given by:
\[ \boxed{(\xi,v)\mapsto(z^{-1},zu)} \]
Consider the bundle  $E$ on $Z_1$, given by transition matrix
\begin{equation}
\left[\begin{array}{cc}
z^{1} & z^{-1}e^u\\
0 & z^{-1}
\end{array}\right].
\end{equation}
Note that this is the same matrix used in  (\ref{matrix}).
Thus $E$ 
 corresponds to the  element  $ z^{-1}e^u \in \Ext^1(\mathcal{O}_{Z_{1}}(1),\mathcal{O}_{Z_{1}}(-1))
 \simeq \HH^1(Z_1,\mathcal{O}_{Z_{1}}(-2)).$
Consequently, we may rewrite the transition function 
\begin{equation}\displaystyle
\left[\begin{array}{cc}
z^{1} & z^{-1}e^u\\
0 & z^{-1}
\end{array}\right]=\left[\begin{array}{cc}
z^{1} & z\sigma\\
0 & z^{-1}
\end{array}\right]
\label{eq2}
\end{equation}
where $z^{-2} u = \sigma\in H^1(Z_1,\mathcal{O}_{Z_{1}}(-2)).$
But $\sigma = \xi^3 v $ is holomorphic on the $V$ chart, and hence a coboundary. Thus 
$\sigma = 0 \in \HH^1(Z_1,\mathcal{O}_{Z_{1}}(-2))$,
and accordingly $z^{-1} e^u = 0 \in \Ext^1(\mathcal{O}_{Z_{1}}(1),\mathcal{O}_{Z_{1}}(-1))$.
Therefore the extension splits and 
$$E = \mathcal{O}_{Z_{1}}(-1) \oplus \mathcal{O}_{Z_{1}}(1) \text{.}$$

\subsection{Deformations of \texorpdfstring{$Z_k$}{Z\_k}}

\cite[Thm.\thinspace 5.3]{BG} construct a $(k-1)$-dimensional semiuniversal deformation space $\mathcal Z$  for $Z_k$  given by
\begin{equation}\label{N}
(\xi, v, t_1, \dotsc, t_{k-1})=(z^{-1}, z^ku + t_{k-1}z^{k-1} + \dotsb + t_1z  , t_1, \dotsc, t_{k-1}) \text{ .}
\end{equation}

We now prove that this family fits our definition of deformation.

\begin{lemma} \label{ZkCinf}
The deformation given by eq. \ref{N} is a $C^\infty$-trivial fiber bundle.
\end{lemma}

\begin{proof}
Note thar for any $C^\infty$ function $f \colon U \to \mathbb{C}$, the manifold given by gluing the charts $V=\mathbb{C}^2_{\xi,v}$ and $U=\mathbb{C}^2_{z,u}$ by
 \[
(\xi, v, t_1, \dotsc, t_{k-1})=(z^{-1}, z^ku + f(z,u))
\]
whenever $z \neq 0$ and $\xi \neq 0$ 
is diffeomorphic to $Z_k$.


We have that $z^{-1}$ and $z^ku$ is $C^\infty$. Then $u$ is $C^\infty$, as well as $\textnormal{Re}(u)$ and $\textnormal{Im}(u)$, respectively the real and imaginary parts of $u$. Hence  
\[
\frac{zu+z\bar{u}}{2\textnormal{Re}(u)} \textnormal{ and } \frac{zu-z\bar{u}}{2i\textnormal{Im}(u)}
\]
are $C^\infty$ and coincide with $z$ whenever $\textnormal{Re}(u)$ and $\textnormal{Im}(u)$ are not equal to 0, respectively. We define then 
\[
f(z,u)=
\begin{cases}
\dfrac{zu+z\bar{u}}{2\textnormal{Re}(u)}, \textnormal{Re}(u) \neq 0 \\ 
\dfrac{zu-z\bar{u}}{2i\textnormal{Im}(u)}, \textnormal{Im}(u) \neq 0 \\
z, u=0
\end{cases},
\]
which is $C^\infty$ on the intersection. Furthermore, $f$ coincides with $z$.

Whe conclude that $g(z,u)=t_{k-1}z^{k-1} + \dotsb + t_1z  , t_1, \dotsc, t_{k-1}$ is $C^\infty$.
\end{proof}

\begin{lemma} \label{ZtDeformations} Deformations of $Z_k$ can be obtained 
from deformations of\\  $ \mathbb{F}_k$. Thus, 
the family $\mathcal{Z}$ is  is not universal. 
\end{lemma}

\begin{proof}
We  compare deformations of the surfaces $Z_k$ with those of the Hirzebruch surfaces.
Choose coordinates
$(t_1, \dotsc, t_{k-1}, [l_0, l_1], [x_0, \dotsc, x_{k+1}])$
for the product  $\mathbb{C}_t^{k-1} \times \mathbb{P}^1_l \times \mathbb{P}^{k+1}_x$.
\cite[Chap.\thinspace II]{M} shows that the Hirzebruch surface  $\mathbb{F}_k$ has a $(k-1)$-dimensional
semiuniversal deformation space given by the smooth subvariety
$M \subset \mathbb{C}_t^{k-1} \times \mathbb{P}^1_l \times \mathbb{P}^{k+1}_x$
cut out  by the equations
\begin{equation}\label{M}
l_0(x_1, x_2, \ldots, x_k)=l_1(x_2 - t_1x_0, \ldots, x_k-t_{k-1}x_0, x_{k+1}) \text{ .}
\end{equation}

Let $\mathcal{Z}$ and $M$ denote the deformations given by \ref{N} and \ref{M}, respectively. 
Now consider the following map: 
\[
\begin{array}{rcl}
f\colon \mathcal{Z} & \to & M\\
(z, u, t_1, \ldots, t_{k-1}) & \mapsto & (t_1, \ldots, t_{k-1}, [1,z], [-1, z_1,\ldots , z_k, u])\\
(\xi, v, t_1, \ldots, t_{k-1}) & \mapsto & (t_1, \ldots, t_{k-1}, [\xi, 1], [-1, v, \xi_2, \ldots,\xi_{k+1} ]) 
\end{array}
\]
where we used the following notation:
\begin{align*}
z_1 &= z^k u + t_{k-1} z^{k-1} + \dotsb + t_1z & \xi_2 &= \xi v-t_1 \\
z_2 &= z^{k-1}u + t_{k-1} z^{k-2} + \dotsb + t_2z & \xi_3 &= \xi^2v - t_1\xi - t_2 \\
    &\ \;\vdots & &\ \;\vdots \\
z_{k-1} &= z^2u+t_{k-1}z & \xi_{k} &= \xi^{k-1}v - t_1\xi^{k-2} - \dotsb - t_{k-1}\\
z_k &= zu & \xi_{k+1} &= \xi^k v- t_1 \xi^{k-1} - \dotsb - t_{k-1}\xi
\end{align*}
It turns out that this map is injective and satisfies $f(\mathcal{Z}_t) \subset M_t$ for all $t \in \mathbb{C}^{k-1}$.
Notice that, for each $t \in \mathbb{C}^{k-1}$, we can  decompose $M_t$ as
\[
M_t = A_t \cup B_t \text{ ,}
\]
where $A_t = \{p \in M_t, x_0=0\}$ and $B_t = \{p \in M_t, x_0 \neq 0\}$. It then  follows that
\begin{itemize}
\item $B_t = f(\mathcal{Z}_t)$, and
\item $A_t$ is the boundary of $B_t$,
\end{itemize}
implying as a corollary that:
$M_t=M_{t'}$ if and only if $\mathcal{Z}_t=\mathcal{Z}_{t'}$.

So we conclude that each $Z_k$ has as many deformations as $\mathbb{F}_k$, specifically, $\lfloor k/2 \rfloor$. In particular, $\mathcal{Z}_k$ is not universal.
 \end{proof}

\section{Deformations of Calabi--Yau threefolds}


\subsection{Rigidity of \texorpdfstring{$W_1$}{W\_1}}
\begin{theorem}\cite{R}\label{w1}
$W_1$ is formally rigid. 
\label{W_1isrigid}
\end{theorem}

\begin{proof}
Formal infinitesimal deformations of complex structures are parameterised by first cohomology with coefficients 
in the tangent bundle. Direct calculation  of  \v{C}ech cohomology shows that $\HH^1(W_1, TW_1) = 0$.
Hence $W_1$ is formally rigid,  Definition \ref{formal}. 
\end{proof}

\subsection{Deformations of \texorpdfstring{$W_2$}{W\_2}}
Since we have $\HH^2(W_2,TW_2)=0$, we can make an analogy with unobstructed deformations in the compact case, where the theorem of existence \cite[Thm.\thinspace 5.6]{Ko} guaranties integrability of the cocycles in $\HH^1(W_2,TW_2)$. 
This theorem does not apply in the noncompact case. 
For the case that we consider, we will prove existence by explicitly constructing the corresponding manifold as Lemma \ref{integrableW2} shows.

It is possible to obtain some deformations using compactifications, in which case we can use the well developed theory of deformations from \cite{Ko}. However, given the results of Theorem \ref{W_2iddef}, infinetely many directions of deformations of $W_2$ would be lost if we worked with the compactification. Hence, we favor an approach using Definition \ref{deformation}.

	For instance, suppose we consider the compactification of $W_2$ given by:
	\[\overline{W}_2=\Proj\bigl(\mathcal O_{\mathbb P^1}(-2)\oplus\mathcal O_{\mathbb P^1}\oplus\mathcal O_{\mathbb P^1}\bigr).\]

	\begin{lemma} $\overline{W}_2$ has only two directions of deformation.
	\end{lemma}
	\begin{proof}
		The first cohomology group of  $\overline{W}_2$ is isomorphic to $\mathbb{C}^2$ as a vector space, that is, 
	$\HH^1(\overline{W}_2, T\overline{W}_2) =\mathbb{C}^2$. 
	\end{proof}

In fact, many non affine deformations would remain unfound with this method.

\begin{theorem}\label{W_2iddef}\cite{R}
$W_2$ has an infinite-dimensional family of deformations.
\end{theorem}
\begin{proof}
The proof will follow from Lemmas \ref{tangentW2} and 
\ref{integrableW2} below. First we show that the first  cohomology with tangent coefficients is 
infinite-di\-men\-sional.
Then we show that its cocycles are integrable, and thus they parametrise deformations of $W_2$.
\end{proof}

\begin{lemma} \label{tangentW2}  $\HH^1(W_2,TW_2)$  is infinite dimensional over $\mathbb{C}$. If consists of holomorphic sections of the form
\[
\sum_{j \geq 0} \left[ \begin{array}{c}
0 \\ z^{-1}u_2^j \\ 0
\end{array} \right]
\]
 (written in canonical coordinates).
\end{lemma}

\begin{proof}
$W_2$ can be covered by
\[ U = \{(z,u_1,u_2)\} \quad\text{and}\quad V = \{(\xi, v_1, v_2)\} \text{ ,} \]
with $U \cap V = \mathbb C-\{0\} \times \mathbb C \times \mathbb C$ and transition function given by:
\[\boxed{ (\xi, v_1, v_2) = (z^{-1}, z^2u_1, u_2)} \]
We have then that the transition function for $TW_2$ is
\[ A= \left[ \begin{array}{ccc}
-z^{-2} & 0 & 0 \\
2zu_1 & z^2 & 0 \\
0 & 0 & 1
\end{array} \right]. \]
Let $\sigma$ be a 1-cocycle, i.e.\ a holomorphic function on $U \cap V$:
\[ \sigma = \sum_{j=0}^{\infty} \sum_{i=0}^{\infty}  \sum_{l=-\infty}^{\infty}
\left[ \begin{array}{c}
a_{lij} \\ b_{lij}\\ c_{lij}
\end{array} \right]
z^l u_1^i u_2^j \text{ .} \]
But
\[ \sum_{j=0}^{\infty} \sum_{i=0}^{\infty}   \sum_{l=0}^{\infty}
\left[ \begin{array}{c}
a_{lij} \\ b_{lij}\\ c_{lij}
\end{array} \right] 
z^l u_1^i u_2^j \]
is a coboundary, so 
\[ \sigma \sim \sum_{j=0}^{\infty} \sum_{i=0}^{\infty}  \sum_{l=-\infty}^{-1}
\left[ \begin{array}{c}
a_{lij} \\ b_{lij}\\ c_{lij}
\end{array} \right]
 z^l u_1^i u_2^j=\sigma',\]
where $\sim$ denotes cohomological equivalence. So
\begin{align*}
A \sigma' & =  \sum_{j=0}^{\infty} \sum_{i=0}^{\infty}  \sum_{l=-\infty}^{-1}
\left[ \begin{array}{c}
-a_{lij}z^{-2} \\ 2a_{lij}zu_1 + b_{lij}z^2\\ c_{lij}
\end{array} \right]
 z^{l} u_1^i u_2^j \\
  & = \sum_{j=0}^{\infty} \sum_{i=0}^{\infty}  \sum_{l=-\infty}^{-1}
\left[ \begin{array}{c}
-a_{lij}z^{-4} \\ 2a_{lij}z^{-3}(z^2u_1) + b_{lij}\\ c_{lij}z^{-2}
\end{array} \right]
 z^{2+l-2i} (z^2u_1)^i u_2^j \\
 & =  \sum_{j=0}^{\infty} \sum_{i=0}^{\infty}  \sum_{l=-\infty}^{-1}
\left[ \begin{array}{c}
-a_{lij}\xi^4 \\ 2a_{lij}\xi^3 v_1 + b_{lij}\\ c_{lij} \xi^2
\end{array} \right]
 \xi^{2i-l-2} v_1^i v_2^j. \\
\end{align*}
Except for the case where $l=-1$ and $i=0$, we have that $2i-l-2 \geq 0$, thus the corresponding 
monomials are holomorphic in $V$ and hence coboundaries.
It follows that
\begin{align*}
A \sigma' & \sim \sum_{j=0}^{\infty}
\left[ \begin{array}{c}
-a_j \xi^4 \\ 2a_j\xi^3 v_1 + b_j\\ c_j \xi^2
\end{array} \right]
\xi^{-1}v_2^j \\
& \sim \sum_{j=0}^\infty
\left[ \begin{array}{c}
0 \\ b_j\\ 0 
\end{array} \right] 
\xi^{-1}v_2^j,
\end{align*}
where we omit the indices $-1$ for $l$ and $0$ for $i$ for simplicity. 
We conclude then that $\HH^1(W_2, TW_2)$ is infinite-dimensional,
generated by the sections
\[
\sigma_j= \left[ \begin{array}{c}
0 \\ z^{-1}u_2^j\\ 0 
\end{array} \right]
\]
for $j \geq 0$.
\end{proof}

\begin{lemma}\label{integrableW2}
All cocycles in $\HH^1(W_2, TW_2)$ are integrable. 
\end{lemma}
\begin{proof}
We can write the transition function of $W_2$ as:
\[
\left[
\begin{array}{c}
\xi \\ v_1 \\ v_2
\end{array}
\right]
=
\left[
\begin{array}{c}
z^{-1} \\ z^2u_1 \\ u_2
\end{array}
\right]
=
\left[
\begin{array}{ccc}
z^{-2} & 0 & 0 \\
0 & z^2 & 0 \\
0 & 0 & 1
\end{array}
\right]
\left[
\begin{array}{c}
z \\ u_1 \\ u_2
\end{array}
\right].
\]
As we computed in Lemma \ref{tangentW2},  $\HH^1(W_2,TW_2)$ is generated by the sections 
\[
\left[
\begin{array}{c}
0 \\ z^{-1}u_2^j \\0
\end{array}
\right]
\]
for $j \geq 0$. Then we can  express the deformation family for $W_2$ as
\begin{align*}
\left[
\begin{array}{c}
\xi \\ v_1 \\ v_2
\end{array}
\right]
& =
\left[
\begin{array}{ccc}
z^{-2} & 0 & 0 \\
0 & z^2 & 0 \\
0 & 0 & 1
\end{array}
\right]
\left(
\left[
\begin{array}{c}
z \\ u_1 \\ u_2
\end{array}
\right]
+ \sum_{j \geq 0} t_j
\left[
\begin{array}{c}
0 \\ z^{-1}u_2^j \\ 0
\end{array}
\right]
\right) \\
& = 
\left[
\begin{array}{c}
z^{-1} \\ z^2u_1 + \sum_{j \geq 0} t_jzu_2^j \\ u_2
\end{array}
\right],
\end{align*}
i.e.\ we have an infinite-dimensional deformation family given by
\[ U = \mathbb{C}^3_{z,u_1,u_2} \times \mathbb{C}[t_j] \quad\text{and}\quad
   V = \mathbb{C}^3_{\xi,v_1,v_2} \times \mathbb{C}[t_j] \]
with
\[ (\xi,v_1,v_2, t_0, t_1, \ldots) = \left(z^{-1},z^2u_1 + \sum_{j \geq 0} t_jzu_2^j ,u_2, t_0, t_1, \dotsc\right) \]
on the intersection $U \cap V = (\mathbb{C}-\{0\}) \times \mathbb{C}^2 \times \mathbb{C}[t_j]$.
\end{proof}

The proof that this family is $C^\infty$ trivial is similar to the proof of Lemma \ref{ZkCinf}.

\subsubsection{A non-affine deformation}

The proof of Lemma \ref{integrableW2}
gives us that deformations of $W_2$ are  threefolds given by  change of coordinates  of the form
\[ (\xi, v_1, v_2) = \left(z^{-1}, z^2u_1 + \sum_{j \geq 0} t_j z u_2^j, u_2\right) \text{ .} \]
We consider now the example  $\mathcal W_2$  that occurs when $t_1=1$ and all $t_j$ vanish for $j \neq 1$, that is, 
the one with change of coordinates
\[
(\xi,v_1,v_2)=\left(z^{-1},z^2u_1 + zu_2  ,u_2\right) .
\]

\begin{lemma}
Let $\mathcal{O}_{\mathcal{W}_2}(-j) = p^\ast(\mathcal{O}_{\mathbb{P}}(-j))$ denote the pullback bundle of $\mathcal{O}_{\mathbb{P}}(-j)$, where $p \colon \mathcal{W}_2 \to \mathbb{P}$ is the natural projection. Then 
$\HH^1(\mathcal{W}_2, \mathcal{O}_{\mathcal{W}_2}(-4)) \neq 0$.
\end{lemma}

\begin{proof} Consider the 1-cocycle $\sigma$ written in the $U$ coordinate chart as 
$\sigma = z^{-1} $.
Suppose $\sigma $ is a coboundary, then we must have 
$$\sigma = \alpha + T^{-1} \beta$$ where $\alpha \in \Gamma(U)$ and $\beta = \Gamma(V)$.
Consequently
\[ z^{-1} =  \alpha(z,u_1,u_2) + z^{-4} \beta(z^{-1}, z^2u_1 + zu_2, u_2) \text{ .} \]
But $\alpha$ has only positive powers of $z$, and the highest power of 
$z$ appearing on $z^{-4} \beta$ is $-4$, hence 
the right-hand side has no terms in $z^{-1}$ and the equation is impossible, a contradiction.
\end{proof}

\begin{corollary}
$\mathcal W_2$ is not affine.
\end{corollary}

\begin{remark}
Note that this result contrasts with the situation for surfaces, since  \cite[Thm.\thinspace 6.15]{BG} 
prove that all nontrivial deformations of $Z_k$ are af\-fine. 
\end{remark}

\begin{remark}
The referee pointed out that all deformations of $W_2$ such that $t_0=0$ are affine since they contain a $\mathbb{P}^1$.
\end{remark}

\subsection{Deformations of \texorpdfstring{$W_3$}{W\_3}}
\noindent
We start by computing the group $\HH^1(W_3,TW_3)$ which parametrises formal infinitesimal deformations of $W_3$.
Recall that $W_3$ can be covered by $U=\{(z,u_1,u_2)\}$ and $V=\{(\xi, v_1, v_2)\}$, 
with $U \cap V = \mathbb C-\{0\} \times \mathbb C^2$ and transition function given by:
\begin{equation}\label{W3}
\boxed{(\xi, v_1, v_2) = (z^{-1}, z^3u_1, z^{-1}u_2)}
\end{equation}

\begin{theorem}
There is a formal versal deformation space $\mathcal{W}$  for $W_3$  parametrised by cocycles of the form
$$
\left[
\begin{array}{cccr}
a_{lij} \\ 
b_{lij} \\
c_{lij}
\end{array}
\right] z^l u_1^i u_2^j 
\qquad  3i-3-l-j<0.
$$
\end{theorem}

\begin{proof}
In canonical coordinates, the transition  matrix for the tangent bundle  $TW_3$ is given by
\begin{equation}\label{transitionW3}
T = \left[ \begin{matrix}
-z^{-2} & 0 & 0 \\
3 z^2 u_1 & z^3 & 0 \\
-z^{-2}u_2 & 0 & z^{-1}
\end{matrix} \right] \simeq 
\left[ \begin{matrix}
z^{-1} & 0 & -z^{-2}u_2 \\
0  & z^3 & 3 z^2 u_1 \\
0 & 0 & -z^{-2}
\end{matrix} \right], 
\end{equation}
where $\simeq$ denotes isomorphism, and the latter expression is handier for calculations. 
A 1-cocycle can be expressed in $U$ coordinates in the form 
\begin{align*}
\sigma & = \sum_{j=0}^\infty  \sum_{i=0}^\infty   \sum_{l=-\infty}^\infty
\left[
\begin{array}{c}
a_{lij} \\ 
b_{lij} \\
c_{lij}
\end{array}
\right] z^l u_1^i u_2^j \\
& \sim \sum_{j=0}^\infty  \sum_{i=0}^\infty   \sum_{l=-\infty}^{-1}
\left[ 
\begin{array}{c}
a_{lij} \\ 
b_{lij} \\
c_{lij}
\end{array}
\right] z^l u_1^i u_2^j, 
\end{align*}
where $\sim$ denotes cohomological equivalence.
Changing coordinates we obtain
\begin{align*}
T \sigma  & = \sum_{j=0}^\infty  \sum_{i=0}^\infty  \sum_{l=-\infty}^{-1}
\left[ \begin{matrix}
a_{lij}z^{-1} - c_{lij}z^{-2}u_2\\
3a_{lij}z^2u_1 + b_{lij}z^3 \\
- c_{lij} z^{-2} \\
\end{matrix} \right] 
z^l u_1^i u_2^j\\
\end{align*}
where all terms inside the matrix are holomorphic on $V$ except for 
  $$
\left[ \begin{matrix}
 0 \\
b_{lij}z^3 \\
 0  \\
\end{matrix} \right] .
$$
These impose the condition for a cocycle to be  nontrivial.
Since we have
\[ z^3 z^l u_1^i u_2^j =  z^{l+3-3i+j}(z^3 u_1)^i (z^{-1}u_2)^j  = \xi^{3i-3-l-j}u_1^i u_2^j \text{ ,} \]
a nontrivial cocycle satisfies $3i - 3 - l - j < 0$.
\end{proof} 

We now give a partial description of deformations of $W_3$.

\begin{lemma}\label{W3cocycles}
The sections
\[
\sigma_1=\left[ \begin{array}{c}
0 \\ z^{-1} \\ 0
\end{array} \right]
\textnormal{ and }
\sigma_2=\left[ \begin{array}{c}
0 \\ z^{-2} \\ 0
\end{array} \right]
\]
are nonzero cocycles on $\HH^1(W_3,TW_3)$.
\end{lemma}

\begin{proof}
Let
\[
\sigma_l = \left[ \begin{array}{c}
o \\ z^{-l} \\ 0
\end{array}\right],
\]
for $l=1, 2$. Then $\sigma_l$ is not a coboundary on the chart $U$.
We change coordinates by multiplying by the transition $T$ given in \ref{transitionW3},
\[
T \sigma_l = \left[ \begin{array}{c}
0 \\ z^{l+3} \\ 0
\end{array} \right]
= \left[ \begin{array}{c}
0 \\ \xi^{-l-3} \\ 0
\end{array} \right],
\]
which is not holomorphic on the chart $V$ and therefore not a coboundary.
\end{proof}

\begin{lemma}\label{W3deformation}
The following 2-parameter family of deformations of $W_3$ is contained in $\mathcal W$:
\[ (\xi, v_1,v_2)=(z^{-1},z^3u_1+t_2z^2+t_1z,z^{-1}u_2) \]
\end{lemma}

\begin{proof}
The transition for $W_3$ is given by, 
\[
(\xi, v_1, v_2)=(z^{-1}, z^3u_1, z^{-1}u_2).
\]
In matrix form:
\[
\left[ \begin{array}{c}
\xi \\ v_1 \\ v_2
\end{array} \right]
=
\left[ \begin{array}{ccc}
z^{-2} & 0 & 0 \\ 0 & z^3 & 0\\ 0 & 0 & z^{-1}
\end{array} \right]
\left[ \begin{array}{c}
z \\ u_1 \\ u_2
\end{array} \right].
\]
So we can construct a deformation family for $W_3$ using the cocycles from Lemma \ref{W3cocycles}:
\begin{align*}
\left[ \begin{array}{c}
\xi \\ v_1 \\ v_2
\end{array} \right]
& =
\left[ \begin{array}{ccc}
z^{-2} & 0 & 0 \\ 0 & z^3 & 0\\ 0 & 0 & z^{-1}
\end{array} \right]
\left(
\left[ \begin{array}{c}
z \\ u_1 \\ u_2
\end{array} \right]
 + 
t_2
\left[ \begin{array}{c}
0 \\ z^{-1} \\ 0
\end{array} \right]
+ t_1
\left[ \begin{array}{c}
0 \\ z^{-2} \\ 0
\end{array} \right]
\right) \\
& = 
\left[ \begin{array}{c}
z^{-1} \\ z^3u_1 + t_2z^2 + t_1z \\ z^{-1}v_2
\end{array} \right]
\end{align*}
Now it suffices to observe that, by Lemma \ref{W3cocycles},
$\sigma_1$ and $\sigma_2$ are nontrivial directions in $\mathcal W$. 
\end{proof}

\begin{corollary} The family presented in Theorem  \ref{W3deformation} is formally semiuniversal but not universal.
\label{W_3sunivnouniv}\end{corollary}

\begin{proof}
As a consequence of Lemma \ref{W3deformation} and Corollary \ref{ZtDeformations}, we have that the deformations in the directions of the cocycles of Lemma \ref{W3cocycles} are isomorphic. Indeed, these deformations are induced by $Z_3$ which, as $\mathbb{F}_3$, only has one nontrivial direction of  deformation.
\end{proof}

\section*{Acknowledgements}
\noindent
We thank professor Edoardo Ballico for enlightening discussions. 

Our results  were presented by Gasparim and Suzuki at the Geometry and Physics session of the 
{\it V Congreso Latinoame\-ri\-cano de Matem\'aticas}. We thank 
 UMALCA, Universidad del Norte  and the Colombian Mathematical Society for the financial support and hospitality, 
 and
   Bernardo Uribe for the invitation to organise a session  as well as for  giving us the opportunity to submit a contribution to these proceedings.
 
Gasparim was partially supported by  the Vicerrector\'ia de Investicagi\'on y Desarrollo tecnol\'ogico at Universidad Cat\'olica del Norte (Chile).
Suzuki acknowledges support from  Beca Doctorado Nacional -- Folio 21160257.
Rubilar acknowledges support from  Beca Doctorado Nacional -- Folio 21170589.
The final version of this work was written during a visit to the ICTP of Rubilar and Gasparim under support of a 
Simmons Associateship grant.

\vfill\noindent
E.\ Gasparim, T.\ K\"oppe, F.\ Rubilar, B.\ Suzuki\\
Departamento de Matem\'aticas\\
Universidad Cat\'olica del Norte\\
Av. Angamos 0600\\
Antofagasta\\
Chile

\end{document}